\documentclass[11pt]{article}
\usepackage{epsf}
\usepackage{color}
\usepackage{epsfig}

\usepackage{graphicx}
\usepackage{amsmath}
\usepackage{latexsym}
\usepackage{amssymb}
\usepackage{amsfonts}

\newtheorem{theorem}{Theorem}

\newtheorem{proposition}[theorem]{Proposition}
\newenvironment{proof}[1][Proof]{\textbf{#1.} }
     {\    \rule{0.5em}{0.5em}}

\newcommand{\calA}{{\mathcal A}}
\newcommand{\calB}{{\mathcal B}}
\newcommand{\calC}{{\mathcal C}}
\newcommand{\E}{{\mathbb E}}
\newcommand{\R}{{\mathbb R}}

\begin{document}
\title{Symmetrically processed splitting integrators for enhanced Hamiltonian Monte Carlo sampling}
\author{
S. Blanes,\footnote{Instituto de Matem\'atica Multidisciplinar, Universitat Polit\`ecnica
de Val\`{e}ncia, 46022-Valencia, Spain. E-mail: serblaza@imm.upv.es}\:
M.P. Calvo,\footnote{Departamento de Matem\'atica Aplicada e IMUVA, Facultad de Ciencias, Universidad de
Valladolid,  Spain. E-mail: mariapaz.calvo@uva.es} \:
F. Casas,\footnote{Departament de Matem\`{a}tiques and IMAC, Universitat Jaume I, E-12071 Castell\'{o}n, Spain. E-mail: casas@uji.es}
\: and J.M. Sanz-Serna\footnote{Departamento de Matem\'aticas, Universidad Carlos III de Madrid, Avenida de la
Universidad 30, E-28911 Legan\'es (Madrid), Spain. E-mail: jmsanzserna@gmail.com}
}
\date{\today}

\maketitle

\begin{abstract}{We construct integrators to be used in  Hamiltonian (or Hybrid) Monte Carlo sampling. The new integrators are easily implementable and, for a given computational budget, may deliver five times as many accepted proposals as standard leapfrog/Verlet without impairing in any way the quality of the samples. They are based on a suitable modification of the   processing technique first introduced by J.C. Butcher. The idea of modified processing may also be useful for other purposes, like the construction of high-order splitting integrators with positive coefficients.}\end{abstract}

\noindent AMS numbers: 65L05, 65C05, 37J05


\section{Introduction}
In this paper we show how to construct symmetrically processed splitting algorithms for efficient Hamiltonian (or Hybrid) Monte Carlo (HMC)  sampling.

HMC is a widely used sampling technique; introduced in the physics literature \cite{duane}, it has become very popular in statistics \cite{neal} and may provide large improvements over alternative approaches \cite{cances}. All the many variants of HMC share the need for integrating numerically a system of Hamiltonian differential equations at each step of the Markov chain \cite{cetraro} and in fact the gradient evaluations required by the numerical integrator dominate the computational cost of obtaining the samples. It is therefore useful to identify existing integrators or to construct new ones that are efficient in HMC sampling. Even though the leapfrog/St\"{o}rmer/\-Verlet method is the integrator usually chosen, it is possible to cut down substantially the computational cost of the integrations \emph{without impairing in any way} the quality of the sampling by using (multistage) \emph{splitting integrators}
which may be implemented as easily as leapfrog and are \emph{time reversible} and \emph{volume preserving}, two essential requirements for HMC use \cite{acta}.

There are several families of possible splitting schemes and each family  includes free parameters. The paper \cite{sergferyo} suggested a methodology for choosing the splitting parameters so as to optimize the efficiency in the HMC setting. As distinct from other works (see \cite{mclachlan,omelyan,predescu,goth,kitaura} among many others), where the choice of parameters is determined by the behaviour of the integrator as the step size \(h\) approaches 0, in \cite{sergferyo} \(h\) is not assumed to be small; it rather ranges in a suitable interval \((0,\bar h)\). This corresponds to the fact that successful HMC simulations operate with rather large values of \(h\) \cite{neal,acta}. Because free parameters are not used to boost the accuracy in the limit \(h\rightarrow 0\), all the integrators constructed in \cite{sergferyo}  are second order. This  implies that these integrators yield \emph{average} energy errors of order four \cite{optimal,acta}. {They are designed to minimize  the energy error in the proposal and to maximize (see \cite{daniel}) the probability of proposals being accepted by the accept/reject mechanism of the sampler.}
Extensive numerical experiments \cite{sergferyo,attia,mario,cedric,man,nishimura, aleardi,daniel,goth} in a variety of applications that range from Bayesian statistics to the Auxiliary Field Quantum Monte Carlo method
and theoretical results presented in \cite{daniel} endorse the soundness of the approach in \cite{sergferyo}. In particular \cite{daniel} shows experimentally that for splitting formulas that use three gradient evaluations per step, the parameter choice in \cite{sergferyo} leads to better HMC sampling than any other parameter choices. The references \cite{elena,tjana} extended the technique in \cite{sergferyo} to modified HMC algorithms that combine the basic idea of HMC with importance sampling.

The idea of \emph{processing}  numerical integrators \cite[Section 3.5]{serfer} is due to J.C. Butcher \cite{butcher}. Given a one-step integrator, sometimes called the \emph{kernel}, a processed integration requires (i) preprocessing  the initial condition, (ii) integrating with the kernel and (iii) postprocessing the solution. Pre and postprocessing should have negligible complexity, so that  the cost of an integration with processing is essentially the cost of integrating with the kernel. The interest of the idea lies in cases where the accuracy of the processed integration is higher than the accuracy the kernel alone would provide. For instance, the processed algorithm may converge with an order \(\nu\) higher  than the order \(\mu\) of the kernel; when this happens the kernel is said to possess \emph{effective order} \(\nu\). Processing did not become popular when it was first suggested, due to the difficulties of its combination with variable time steps. It reappeared \cite{lm1,lm2,procsfer} in the geometric integration scenario, where the emphasis is in constant time steps \cite{mp}. Since processing may considerably increase the efficiency of an integrator, it is natural to study whether the splitting kernels for HMC applications successfully used in \cite{sergferyo,attia,mario,cedric,man,nishimura, aleardi,daniel,goth} may be processed. Unfortunately the standard approach to processing, where the postprocessing map just inverts the action of the preprocessor, will not work, as it leads to integrations that are \emph{not} time reversible.

In this paper:
\begin{enumerate}
\item We introduce (Section \ref{quasi}) \emph{symmetric (modified) processing}, a modification of standard processing under which time reversible kernels provide time reversible integrations.
\item We provide (Section \ref{choosing}) a methodology to determine the parameters in the kernel and the pre and postprocessor to optimize the performance of the integrator in HMC sampling. This methodology extends the material in \cite{sergferyo}.
\item We construct (Section \ref{meth}) specific symmetric processed algorithms to be applied within HMC simulations.
\item We show (Section \ref{results}) by means of numerical experiments that the sampling efficiency of the new symmetric processed integrators may improve on the standard velocity Verlet integrator by a factor of five or more.
\end{enumerate}

HMC sampling is not the only application where symmetric modified processing may be useful and the final Section~\ref{final} briefly discusses another possible application area: the construction of higher-order splitting methods with positive coefficients. {{} As an illustration we show how to process the well-known Rowlands integrator \cite{rowlands} to get fourth-order integrations while only using substeps with positive coefficients. }

\section{Symmetric (modified) processing}
\label{quasi}
\subsection{Definition}
 Given a system of differential equations \((d/dt)x = f(x)\) in \(\R^D\), each one-step integrator is specified by a map \(\psi_h: \R^D\rightarrow \R^D\) that advances the numerical solution over a time-interval of length \(h\). For instance
\(\psi_h(x) = x+hf(x)\) corresponds to Euler's rule. If \(\psi_h\) is an integrator (sometimes called the kernel) and \(\pi_h:\R^D\rightarrow \R^D\) is a map, one may consider the corresponding so-called \emph{processed integrator} defined by the map
\[
\widehat{\psi}_h = \pi_h^{-1}\circ \psi_h\circ \pi_h
\]
(the superscript \({}^{-1}\) denotes inverse map and \(\circ\) means composition).
 \(N\) consecutive steps of the processed integrator correspond to the transformation
 \[
 \widehat{\psi}_h^N = \overbrace{\widehat{\psi}_h \circ \cdots \circ \widehat{\psi}_h}^{N\: {\rm times}}=
 \overbrace{(\pi_h^{-1}\circ \psi_h\circ \pi_h) \circ \cdots \circ (\pi_h^{-1}\circ \psi_h\circ \pi_h)}^{N\: {\rm times}},
 \]
 that is,
 \begin{equation}\label{eq:proc}
 \widehat{\psi}_h^N = \pi_h^{-1} \circ \psi_h^N \circ\pi_h.
 \end{equation}
 In other words, to perform \(N\) steps of the processed method, one successively  (i) applies once the map \(\pi_h\) (preprocessing), (ii) takes \(N\) steps of the kernel \(\psi_h\) and (iii) applies once the map \(\pi_h^{-1}\) (postprocessing). Since \(\pi_h\) and its inverse are applied only once per integration leg, the computational complexity  of \(\widehat{\psi}_h\) is  not very different from that of \(\psi_h\). Processing is advantageous in situations, among others,  where  \(\widehat{\psi}_h\) is more accurate than the unprocessed \(\psi_h\) (for instance \(\widehat{\psi}_h\) may have higher order of convergence or smaller error constants than \(\psi_h\)).

Recall that the true solution flow \(\phi_t\) of the system being integrated is \emph{time reversible (or symmetric)}, in the sense that  \(\phi_t^{-1}\)
 (which maps the final state into the initial condition) coincides with \(\phi_{-t}\) (which moves the initial condition backwards in time). Correspondingly, an integrator \(\psi_h\) is said to be time reversible (or symmetric) \cite[Section 3.6]{ssc} if \((\psi_h^N)^{-1}= \psi_{-h}^N\). From \eqref{eq:proc} it is easily concluded that, even if the kernel \(\psi_h\) is time reversible, the processed \(\widehat{\psi}_h\) may not be expected to be so.
The symmetric (modified) processing approach suggested in the present paper is a modification of the idea of processing that makes it possible to obtain time reversible integrators.

  To perform an integration leg spanning a time interval of length \(Nh\) with a symmetric modified processed integrator one applies the map
 \begin{equation}\label{eq:quasi}
 \widetilde{\psi}_{N,h} = \pi_h^\star \circ \psi_h^N \circ\pi_h,
 \end{equation}
 where \(\pi_h^\star\) denotes the \emph{adjoint} of \(\pi_h\), i.e.\ the map such that \(\pi_{-h}^\star = \pi_h^{-1}\) (see e.g.\ \cite[Section 3.6]{ssc} or \cite[Section 1.2]{serfer}). This differs from standard processing (see \eqref{eq:proc}) in that the adjoint rather than the inverse is used as a postprocessor.
 Since
 \[
 (\widetilde{\psi}_{N,h})^{-1} = \pi_h^{-1} \circ (\psi_h^N)^{-1} \circ(\pi_h^\star)^{-1}
 \]
 and
 \[
 \widetilde{\psi}_{N,-h} =  \pi_{-h}^\star \circ \psi_{-h}^N \circ\pi_{-h} =\pi_h^{-1} \circ \psi_{-h}^N\circ(\pi_h^\star)^{-1},
 \]
 the symmetric modified processed integrator \eqref{eq:quasi} will be time reversible if \(\psi_h\) is time reversible.

 For processing integrators, the map \eqref{eq:proc} that advances the solution over a time interval \([0,Nh]\) is the \(N\)-th power of the one-step map \(\widehat{\psi}_h\). For symmetric processing the map \(\widetilde{\psi}_{N,h}\) in \eqref{eq:quasi} is not obtained by means of a similar \(N\)-fold composition of a one-step map.

 \subsection{The case of splitting integrators}
 Although the idea of symmetric processing is completely general,
 our attention is restricted to the case where in \eqref{eq:quasi} \(\psi_h\) and \(\pi_h\) are constructed via splitting (see e.g.\ the monographs \cite{ssc,hairer,serfer} and the survey \cite{robert}). If the system being integrated may be written in the split form
 \[
 \frac{d}{dt} x= f(x) = f^{A}(x)+f^{B}(x),
 \]
 and \(\phi_t^A\) and \(\phi_t^B\) represent the exact flows of the split systems,
 we deal with  splitting kernels
 \begin{equation}\label{eq:splitting}
 \psi_h = \phi_{b_1h}^{B}\circ \phi_{a_1h}^{A}\circ \cdots \circ\phi_{a_{r-1}h}^{A}\circ \phi_{b_{r}h}^{B}\circ
 \phi_{a_rh}^{A}
 \circ \phi_{b_{r}h}^{B}\circ\phi_{a_{r-1}h}^{A}\circ\cdots \circ \phi_{a_1h}^{A}\circ \phi_{b_1h}^{B}.
 \end{equation}
 (Since it is possible to set \(a_r=0\), this format includes integrators where the central flow is
 \(\phi^B\) rather than \(\phi^A\). Similarly one may set \(b_1=0\) to have integrators where the extreme flows are \(\phi^A\).) The palindromic structure of \eqref{eq:splitting} ensures time reversibility. This integrator is consistent (in fact of order \(\geq 2\) due to symmetry) if
 \begin{equation}\label{eq:cons0}
 2a_1 +\cdots +2a_{r-1}+a_r= 1,\qquad 2b_1 +\cdots +2b_{r-1}+2b_r= 1,
 \end{equation}
 and in what follows we always assume that these conditions hold.

 Similarly, we choose \(\pi_h\) to be a composition of \(2s\) flows of the form
 \begin{equation}\label{eq:pi}
 \pi_h =\phi_{c_sh}^{A} \circ \phi_{d_sh}^{B}\circ \cdots \circ \phi_{c_1h}^{A} \circ \phi_{d_1h}^{B},
 \end{equation}
 which leads to
  \begin{equation}\label{eq:pistar}
 \pi_h^\star = \phi_{d_1h}^{B}\circ \phi_{c_1h}^{A}\circ \cdots\circ \phi_{d_sh}^{B}\circ \phi_{c_sh}^{A}.
 \end{equation}
 We assume that
 \begin{equation}\label{eq:cons}c_1+\dots +c_s = 0,\qquad d_1+\dots +d_s = 0,
  \end{equation}
  which imply that \(\pi_h\) and \(\pi_h^\star\) differ from the identity map by \(\mathcal{O}(h^2)\) terms. In this way, it is clear that \eqref{eq:quasi} is a time reversible integrator with even order of accuracy \(\geq 2\).

 We use the abbreviations
\[(b_1,a_1,\dots,a_{r-1},b_{r},a_r,b_{r},a_{r-1}, \dots, a_1,b_1),\]
%
and \[
 (c_s,d_s, \dots, c_1,d_1),\qquad (d_1,c_1,\dots,d_s,c_s)
 \]
 to refer to \eqref{eq:splitting}, \eqref{eq:pi}, and \eqref{eq:pistar} respectively. In this way, \eqref{eq:quasi} is denoted as
 \begin{equation}\label{eq:ab}
 (d_1,c_1,\dots,d_s,c_s)(b_1,a_1\dots,b_{r},a_r,b_{r}, \dots, a_1,b_1)^N(c_s,d_s, \dots, c_1,d_1);
 \end{equation}
 the palindromic structure is apparent.

 \subsection{The Hamiltonian case}
 We have in mind the integration of  Hamiltonian systems of the form
 \begin{equation}\label{eq:ham}
 \frac{d}{dt} q = M^{-1}p,\qquad \frac{d}{dt} p = -\nabla V(q),
 \end{equation}
 where \(M\) is a symmetric, positive definite \(d\times d\) mass matrix and \(V\) denotes the potential. Under the familiar \(q/p\) or potential/kinetic splitting, the split flow \(\phi_t^A\) is the solution flow of the system
 \[
  \frac{d}{dt} q = M^{-1}p,\qquad \frac{d}{dt} p = 0
 \]
 and \(\phi_t^B\) corresponds to
 \[
  \frac{d}{dt} q = 0,\qquad \frac{d}{dt} p = -\nabla V(q).
 \]
In molecular dynamics the transformations \(\phi_t^A\) and \(\phi_t^B\) are known as drifts and kicks respectively. Thus an integration leg with the symmetric processed integrator \eqref{eq:quasi} is a succession kick, drift, kick, \dots, kick with a  palindromic pattern and therefore time reversible {in the sense used above, i.e.\ \((\widetilde \psi_{N,h})^{-1} = \widetilde\psi_{N,-h}\). This is equivalent to its being \emph{time reversible with respect to momentum flipping} \cite{acta}, i.e.\
\(\widetilde \psi_{N,h}(q,p) =(q^\prime,p^\prime)\) implies \(\widetilde \psi_{N,h}(q^\prime,-p^\prime) =(q,-p)\). (Note that changing \(t\) into \(-t\) and \(p\) into \(-p\) leaves \eqref{eq:ham} invariant.)
} In addition the map \eqref{eq:quasi} is volume preserving as a composition of kicks and drifts. Time reversibility and volume preservation ensure that \eqref{eq:quasi} may be used in HMC applications with the simple standard recipe for the accept/reject probability applied with leapfrog \cite{acta}.

 \section{Choosing the parameters  in HMC applications}
 \label{choosing}
 \begin{table}
\hrule
\bigskip
{{}

Given $q^{(0)}\in\mathbb{R}^d$, $m_{\max} \geq 1$, set $m=0$.

\begin{enumerate}
\item (Momentum refreshment.) Draw $p^{(m)}\sim N(0,M)$.

\item (Integration leg.) Compute $(q^*,p^*)$  ($q^*$ is the proposal) by  integrating, by means of
a reversible, volume-preserving integrator with step-size \(h\), the Hamiltonian system \eqref{eq:ham} over an interval  \(0\leq t\leq Nh\). The initial condition is \((q^{(m)},p^{(m)})\).

\item (Accept/reject.) Calculate \[a^{(m)} = \min\big(1, \exp(H(q^{(m)},p^{(m)})-H(q^*,p^*))\big)\]
and draw $u^{(m)} \sim U(0,1)$. If $a^{(m)}>u^{(m)}$, set $q^{(m+1)}=q^*$ (acceptance); otherwise set $q^{(m+1)} = q^{(m)}$ (rejection).

\item Set $m= m+1$. If $m = m_{\max}$ stop; otherwise go to step 1.
\end{enumerate}
\hrule
\caption{HMC algorithm.  The function \(H =(1/2) p^TM^{-1}p+V(q)\) is the Hamiltonian.
 The algorithm generates a Markov chain $q^{(0)}\mapsto q^{(1)} \mapsto \dots \mapsto q^{(m_{\max})}$ reversible with respect to the target probability distribution \(\propto \exp(-V(q))\).}
 \label{tab:alg1}
}\end{table}

 {{}The basic HMC algorithm is summarized in Table~\ref{tab:alg1}. Even though other possibilities exist, we here take the view that the length \(Nh\) of each integration leg is a quantity of order one, to ensure that the proposal is sufficiently far from the current state in order to decrease the correlation of the Markov chain \cite{acta}. Once a suitable value of the combination \(Nh\) has been identified (typically by numerical experiment), the value of \(h\) has to be chosen small enough for the energy error \(H(q^{(m)},p^{(m)})-H(q^*,p^*)\) to be small so as to ensure a reasonable empirical rate of acceptance, see \cite{optimal}. }

 The paper \cite{sergferyo} suggested a technique to identify \lq\lq good\rq\rq\ parameter values in families of numerical integrators for the Hamiltonian system \eqref{eq:ham} in the HMC context.
 Extensive numerical  experiments in statistical and molecular dynamics problems reported in \cite{cedric,mario,daniel} show clearly the soundness of the approach in \cite{sergferyo} and it is this approach that we follow here. However the material in that paper   cannot be directly applied to symmetric modified processed algorithms and in this section we present the necessary modifications.

 The technique in \cite{sergferyo} is based on discriminating between integrators by applying them to the one-degree-of-freedom model Hamiltonian \(p^2/2+q^2/2\) for which the equations of motion correspond to the standard harmonic oscillator and the target probability density function is \(\propto\exp(-(p^2+q^2)/2)\) so that \(q\) and \(p\) are independent with a standard normal distribution. This is similar to the well-established idea of discriminating between  integrators for stiff systems of differential equations by applying them to the simple scalar equation \(dy/dt = \lambda y\). {With the methodology in \cite{sergferyo} integrators are constructed by imposing that they have small energy errors when applied to the standard harmonic oscillator. It may be proved rigorously by diagonalization that the methods constructed in this way also yield small energy errors for general multivariate Gaussian targets.
 In addition, numerical experiments show that these integrators also behave well when applied to arbitrary target probability distributions; this is particularly true for targets in high dimensions, where a central limit theorem applies \cite{daniel}.}

  For the harmonic oscillator, \(N\) steps of a given palindromic splitting integrator \eqref{eq:splitting} define a linear transformation \((q_0,p_0) \mapsto (q_N,p_N)\) of the form
 \begin{equation}\label{eq:linear}
 \left[\begin{matrix}q_N\\p_N\end{matrix}\right] =
  \left[\begin{matrix} C & \chi S\\-\chi^{-1}S&C\end{matrix}\right]
  \left[\begin{matrix}q_0\\p_0\end{matrix}\right].
 \end{equation}
 Here \((q_0,p_0)\) is the initial condition, \((q_N,p_N)\) the numerical solution at the end of the integration leg,  \(C\) and \(S\) are abbreviations for \(\cos(N\theta_h)\) and \(\sin(N\theta_h)\) respectively and  \(\chi=\chi_h\), \(\theta=\theta_h\) are quantities that change with the step length \(h\). A priori, \(\theta_h\) may be complex-valued, but if \(h\) is of sufficiently small magnitude, then
 the \(2\times 2\) matrix above is power bounded (stability) and, as shown in \cite{sergferyo}, this corresponds to \(\theta_h\) being real, something we assume hereafter. (In fact the format \eqref{eq:linear} is not specific to splitting integrators; it is shared by any reasonable time reversible, volume preserving unprocessed integrator for \eqref{eq:ham}, see \cite{sergferyo,acta}.)
 If \(N\) varies, the points \((q_N,p_N)\) defined by \eqref{eq:linear} move on an ellipse of the \((q,p)\)-plane, whose eccentricity is governed by \(\chi_h\). For \(\chi_h=1\), the transformation in \eqref{eq:linear} is a rotation, the ellipse becomes a circle and the energy error in the numerical simulation \((1/2)(q_N^2+p_N^2) - (1/2)(q_0^2+p_0^2)\) vanishes: all proposals are then accepted, {{} see Step 3 in Table~\ref{tab:alg1}}.

 Similarly to \eqref{eq:linear}, for the standard harmonic oscillator, the preprocessor \(\pi_h\) and the postprocessor \(\pi_h^\star\) in \eqref{eq:pi} and \eqref{eq:pistar} are respectively associated with \(2\times 2\) matrices of the form
 \begin{equation}\label{eq:linearpro}
 \left[\begin{matrix}
\alpha& \beta\\
\gamma&\delta
\end{matrix}\right],
\qquad
\left[\begin{matrix}
\delta& \beta\\
\gamma&\alpha
\end{matrix}\right],
 \end{equation}
 where \(\alpha=\alpha_h\), \(\beta=\beta_h\), \(\gamma=\gamma_h\), \(\delta=\delta_h\) are polynomials in \(h\), with \(\alpha_h\delta_h -\beta_h\gamma_h=1\) by conservation of volume. In addition \(\alpha_h\) and \(\delta_h\) are even in \(h\) and \(\beta_h\) and \(\gamma_h\) are odd so that
 \[
 \left[\begin{matrix}
\delta_{-h}& \beta_{-h}\\
\gamma_{-h}&\alpha_{-h}
\end{matrix}\right]
 =
 \left[\begin{matrix}
\delta_{h}& -\beta_{h}\\
-\gamma_{h}&\alpha_{h}
\end{matrix}\right]
=
\left[\begin{matrix}
\alpha_h& \beta_h\\
\gamma_h&\delta_h
\end{matrix}\right]^{-1},
 \]
 as needed for a map and its adjoint. (It is perhaps useful to observe that in standard processing  \eqref{eq:proc}, the matrix of the postprocessor would be
 \[
 \left[\begin{matrix}
\alpha_h& \beta_h\\
\gamma_h&\delta_h
\end{matrix}\right]^{-1} = \left[\begin{matrix}
\delta_h& -\beta_h\\
-\gamma_h&\alpha_h
\end{matrix}\right],
 \]
 which differs from the postprocessing matrix for symmetric processing given in \eqref{eq:linearpro} in the sign of the non-diagonal entries; these entries are \(\mathcal{O}(h)\) as \(h\rightarrow 0\).)

Combining \eqref{eq:linear} with \eqref{eq:linearpro}, one integration leg for \eqref{eq:quasi}  is then given  by
 \begin{equation}\label{eq:lineartres}
 \left[\begin{matrix}q_N\\p_N\end{matrix}\right] =
  \left[\begin{matrix} \calA & \calB\\\calC&\calA\end{matrix}\right]
  \left[\begin{matrix}q_0\\p_0\end{matrix}\right],
 \end{equation}
 with
 \[
 \left[\begin{matrix} \calA & \calB\\\calC&\calA\end{matrix}\right]=
\left[\begin{matrix}
\delta& \beta\\
\gamma&\alpha
\end{matrix}\right]
\left[\begin{matrix} C & \chi S\\-\chi^{-1}S&C\end{matrix}\right]
\left[\begin{matrix}
\alpha& \beta\\
\gamma&\delta
\end{matrix}\right].
 \]
 We multiply out the matrices to find:
 \begin{eqnarray*}
 \calA &=& C (\alpha\delta+\beta\gamma)+S(\gamma\delta\chi-\alpha\beta\chi^{-1}), \\
 \calB&=& C(2\beta\delta)+S(\delta^2\chi-\beta^2\chi^{-1}), \\
 \calC&=&  C(2\alpha\gamma)+S(\gamma^2\chi-\alpha^2\chi^{-1}).
 \end{eqnarray*}
 The \lq\lq ideal\rq\rq\ preprocessor has \(\alpha_h = \chi_h^{1/2}\), \(\delta_h = 1/\alpha_h\), \(\beta_h=\gamma_h=0\) leading to \(\calA = C\), \(\calB = S\), \(\calC = -S\). {{} This processing is ideal because then  \eqref{eq:lineartres} is a rotation,   the symmetrically processed integrator conserves energy exactly and there are no rejections in HMC sampling.} Unfortunately such ideal preprocessor cannot be realized by means of a splitting formula of the form \eqref{eq:pi} and our aim now is to identify processors so that \eqref{eq:lineartres} is, in a suitable sense, as close to a rotation as possible.

 At this stage, we assume that \(q_0\) and \(p_0\) are independent random variables with standard normal distribution (i.e.\ that the Markov chain is at stationarity) and consider the change in energy
 \[
 \Delta(q_0,p_0) = \frac{1}{2} (q_N^2+p_N^2) - \frac{1}{2} (q_0^2+p_0^2)
 \]
 over one integration leg. Here \((q_N,p_N)\) and therefore \(\Delta(q_0,p_0)\) are deterministic functions of the random initial condition and therefore random variables themselves. By conservation of energy \(\Delta(q_0,p_0)\) would vanish if the integrator were exact; in HMC simulations small values of \(\Delta\) correspond to high acceptance probability. {{}In fact, it is proved in \cite[Theorem 1]{daniel} that minimizing the expected energy error is equivalent to maximizing the expected acceptance rate at stationarity.}  We have the following result:
 \begin{proposition}With the preceding notation,
 \(\E(\Delta) = (1/2)(\calB+\calC)^2
 \).
 \end{proposition}
 \begin{proof}
 An elementary computation yields:
\[
2\Delta =
(\calA^2+\calC^2-1)q_0^2+2(\calA\calB+\calC\calA)q_0p_0+(\calB^2+\calA^2-1)p_0^2
\]
and therefore taking expectations
\[
2\:\E(\Delta) = 2\calA^2+\calB^2+\calC^2-2.
\]
The result follows because \(\calA^2 -\calB\calC = 1\) by conservation of volume.
 \end{proof}

 Note that \(\E(\Delta) \geq 0\), a well-known fact in HMC simulations \cite{optimal}. The proposition leads to our next result that extends to the situation at hand the bound in \cite[Proposition 4.3]{sergferyo} valid for the unprocessed case.
\begin{proposition}The expectation of \(\Delta\) may be bounded above as follows:
 \[
\E(\Delta)\leq \rho_h,
\]
with
\[\rho_h =
2 (\alpha\gamma+\beta\delta)^2+\frac{1}{2}\big[(\delta^2+\gamma^2)\chi-(\alpha^2+\beta^2)\chi^{-1}\big]^2
.
\]
\end{proposition}

\begin{proof}From the expressions for \(\calB\) and \(\calC\)
\[
(\calB+\calC)^2=\Big( 2C(\alpha\gamma+\beta\delta)+
 S\big[(\delta^2+\gamma^2)\chi-(\alpha^2+\beta^2)\chi^{-1}\big]\Big)^2.
\]
In the right hand-side we have the standard dot product  of the  vector \({\bf v}_1\in\R^2\) with components
\(2(\alpha\gamma+\beta\delta)\) and \((\delta^2+\gamma^2)\chi-(\alpha^2+\beta^2)\chi^{-1}\) and the \emph{unit}  vector \({\bf v}_2\in\R^2\)
with components \(C=\cos(N\theta_h)\) and \(S=\sin(N\theta_h)\). {{} Invoking the Cauchy-Schwarz inequality,} the magnitude of the inner product can then be bounded above by the length of \({\bf v}_1\) and the result follows easily.
\end{proof}

For a fixed symmetric processed integrator, the upper bound  \(\rho_h\) is \emph{independent of the number of steps} \(N\); it only depends on \(h\) and does so through the quantity \(\chi\) associated with the kernel and with the quantities \(\alpha\), \(\beta\), \(\gamma\), \(\delta\) associated with the pre and postprocessor. Of course, in the case where a \emph{family} of integrators is considered, for each \(h\) the quantity \(\rho_h\)  changes with the specific choice of algorithm within the family. According to the strategy  in \cite{sergferyo}, one should pick up a value \(\bar h\) that represents the maximum value of \(h\) to be used in the simulations of the model problem and then
prefer the member of the family that \emph{minimizes} the expected energy-error metric
\[
\| \rho \|_{\bar h} = \max_{0<h<\bar h} \rho_h.
\]
{Thus, for a given family of integrators depending on parameters \(a\), \(b\), \(c\), \dots, the parameter values are determined by minimizing \(\| \rho \|_{\bar h}\), a quantity related to the performance of the integrators when they are applied to the model harmonic problem. This minimization is performed once and for all, i.e.\ the parameter values found in this way are used for all target distributions.
}

In \cite{sergferyo} it is recommended to set \(\bar h\) equal to the number of evaluations of \(\nabla V\) necessary to perform a single time step. Note that this implies that making the integrator more computationally intensive by increasing the number of gradient evaluations per step increases the value of \(\| \rho \|_{\bar h}\). (It is also possible  \cite{mario} to adapt the value of \(\bar h\) to the specific Hamiltonian under consideration, but that line of thought will not be pursued here.)

\section{Specific integrators}
\label{meth}
Several specific (unprocessed) integrators were constructed in \cite{sergferyo} by means of the methodology introduced there. In the next section we will report numerical results for a method (to be referred to as BlCaSa)  of the form \eqref{eq:splitting} with \(r=2\). A single time step of BlCaSa uses four evaluations of \(\nabla V\), but, since the evaluation of \(\nabla V\) at the first kick of the next time step coincides with the evaluation at the last kick of the current step, the computational cost is essentially three gradient evaluations per time step. Integrators that use two or four evaluations per time step were also constructed in \cite{sergferyo}, but that reference found three evaluations per time step to be preferable. With two evaluations,
the resulting larger value \(\| \rho \|_{\bar h}\) offsets the benefit of the smaller computational cost per time step. Four evaluations lead to a marginal improvement of
\(\| \rho \|_{\bar h}\), which does not really compensate for the extra complication. In \cite{daniel} BlCaSa was found to clearly outperform in HMC applications other integrators of the family \eqref{eq:splitting} with \(r=2\). For these reasons we will use BlCaSa as a measuring rod to assess the efficiency of the symmetric processed integrators to be constructed.

We focus our attention on symmetric processed integrators of the form \eqref{eq:ab} with \(r=2\). After imposing the consistency requirement \eqref{eq:cons0},
we may regard \(a= a_1\) and \(b= b_2\) as free parameters for the kernel. It is shown in \cite{cedric} that unless \(a=b/(6b-1)\) the stability interval of the kernel is very short, which makes the integrator uncompetitive. Therefore we impose this relation and  deal with a one-parameter family of kernels. To keep pre and postprocessing as simple as possible, we set \(s=2\), the lowest value for which the consistency relations \eqref{eq:cons} have a nontrivial solution.
We thus work with a three parameter family of  integrators of the form
\[
   (d, c, -d, -c) \, \left(\frac{1}{2} - b, a, b, 1-2a, b, a, \frac{1}{2}-b \right)^N \, (-c, -d, c, d),
\]
where \(a=b/(6b-1)\).
An integration leg requires a total of \(3N+5\) gradient evaluations (two of them within the pre or postprocessing). Of course, if \(N\) is large this is approximately \(3N\).

\begin{table}[t]
\begin{center}
\begin{tabular}{|l||l|l|l||c||c|}
 $\bar h$ &    $b$  & $c$  &  $d$  &  $\| \rho \|_{\bar h}$ &  $h_s$ \\ \hline
 $3$     & $0.381120$ & $-----$ & $----$ & $7 \times 10^{-5}$ &$4.662$\\ \hline
  $3$    &  $0.348674$  &  $-0.075640$  & $0.069720$     &  $6 \times 10^{-8}$ & $4.985$ \\
  $3.5$  & $0.346660$ & $-0.079510$  &  $0.070171$  &  $5 \times 10^{-7}$  & $5.010$ \\
  $4$    & $0.343684$ & $-0.084690$  &  $0.071880$  &  $5 \times 10^{-6}$  & $5.048$ \\
   $4.5$ & $0.340200$ & $-0.093500$  &  $0.072800$  &  $5 \times 10^{-5}$  &  $5.095$ \\
   \hline
 \end{tabular}
 \end{center}
 \caption{\small The last four rows give information on symmetric processed integrators for HMC applications. The methods have been found by minimizing the expected energy-error metric $\| \rho \|_{\bar h}$ for the values of \(\bar h\) displayed in the leftmost column. The first row corresponds to the BlCaSa integrator. The table provides the parameter values to run the integrators and, in the last column, the length of the linear stability interval of the kernel.}
 \label{table}
 \end{table}

{{}As discussed above, once \(\bar h\) has been chosen, we determine the values of the parameters \(b\),  \(c\), \(d\) by minimizing \(\| \rho \|_{\bar h}\). We use the following procedure. For fixed \(b\), \(c\), \(d\), we study
\(\rho_h\) as a function of \(h\) to check whether the largest \(\rho_{\max}\) of the local maxima of \(\rho_h\)  in the open interval \(0<h<\bar h\) does not exceed \(\rho_{\bar h}\), the value at the upper end of the interval. Once an initial set of parameter values with \(\rho_{\max}\leq \rho_{\bar h}\) has been found by trial and error, we use continuation in \(b\), \(c\), \(d\) to improve the value of \(\rho_{\bar h}\) while checking that \(\rho_{\bar h}\geq \rho_{\max}\).  Typically several sets of values of \(b\), \(c\), \(d\)  exist leading  to essentially the same value of
\(\| \rho \|_{\bar h}\).
}

Initially  we set \(\bar h = 3\), the choice suggested in \cite{sergferyo}. The values  we obtained may be seen in Table~\ref{table}, along with the length \(h_s\) of the stability interval of the kernel, i.e.\ the supremum of the step sizes \(h\) for which the matrix in \eqref{eq:linear} may be bounded independently of \(N\)  thus guaranteeing that errors do not grow exponentially as \(N\) increases. (Note that the stability of the kernel determines the stability of the overall integrator, since the pre and postprocessor are applied only once.) For comparison we have also included in the first row of the table the data corresponding to BlCaSa: symmetric processing results in a reduction of \(\| \rho \|_{3}\) \emph{by three orders of magnitude}. This reduction is achieved at the price of only four additional gradient evaluations per integration leg. {{} (Values  of \(\| \rho \|_{\bar h}\) reported in the table have been rounded above, and are therefore upper bounds.) }

The extremely small value of \(\| \rho \|_{3}\)   that may be achieved in this way, prompted us to explore larger values of \(\bar h\), so as to obtain integrators meant to operate with larger step sizes. Our results with
\(\bar h = 3.5\), \(4\) or \(4.5\) are also reported in the table. The integrator specified in the last row, run with \(0<h<4.5\), is (on the Gaussian model) as accurate as BlCaSa when run with \(0<h<3\).

\section{Numerical results}
\label{results}
\begin{figure}[p]
\includegraphics[width=0.48\hsize]{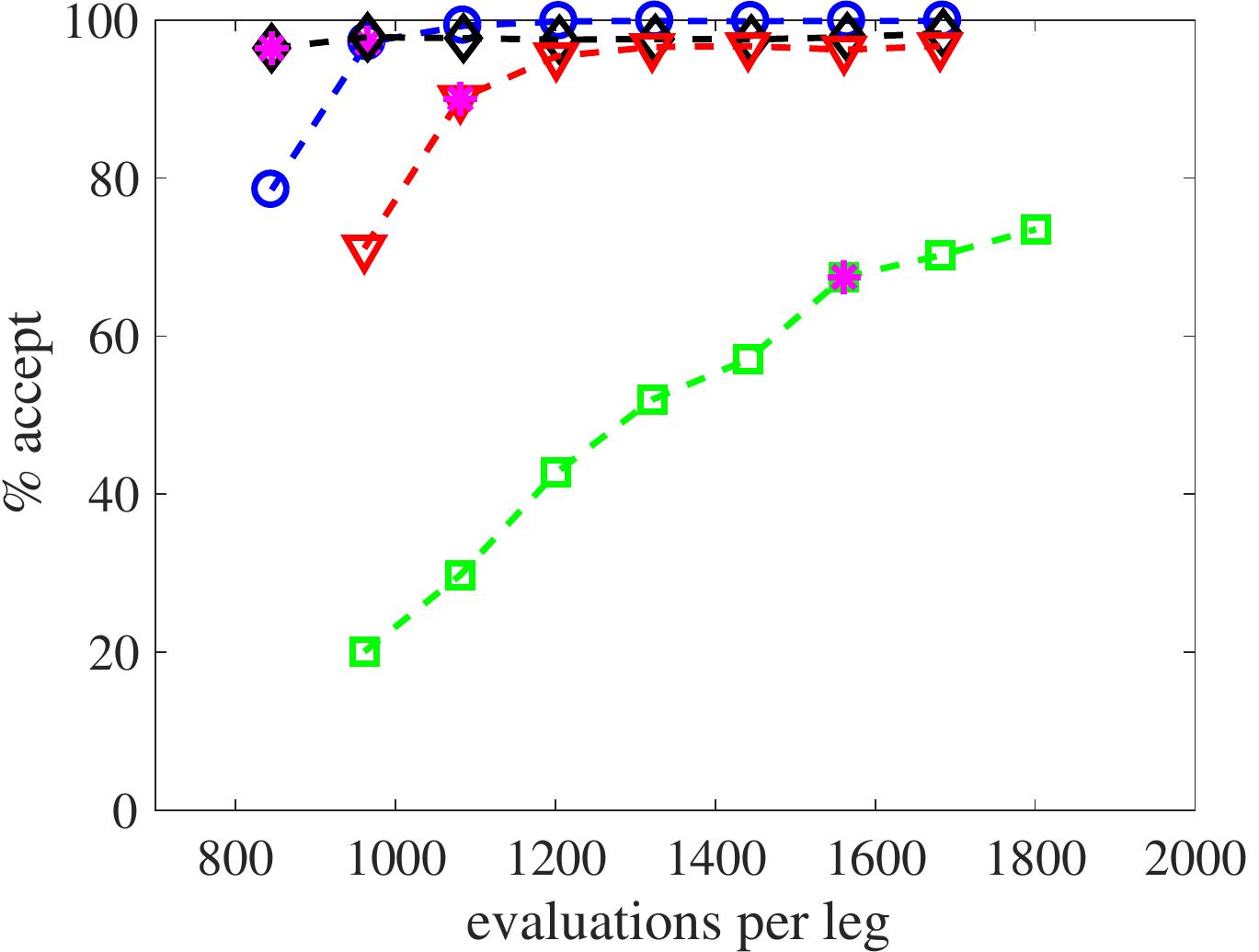}\quad
\includegraphics[width=0.48\hsize]{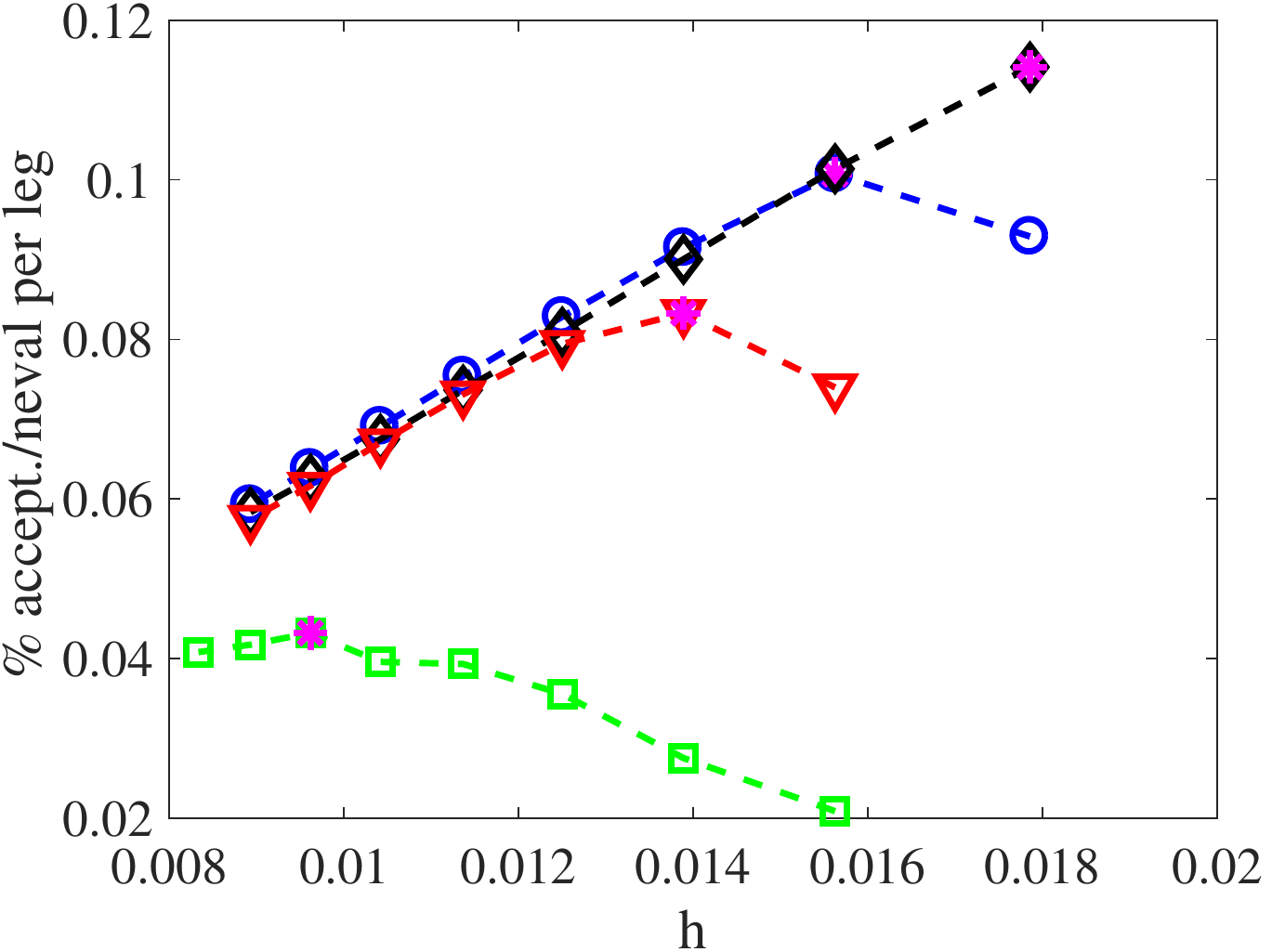}
\\
\includegraphics[width=0.48\hsize]{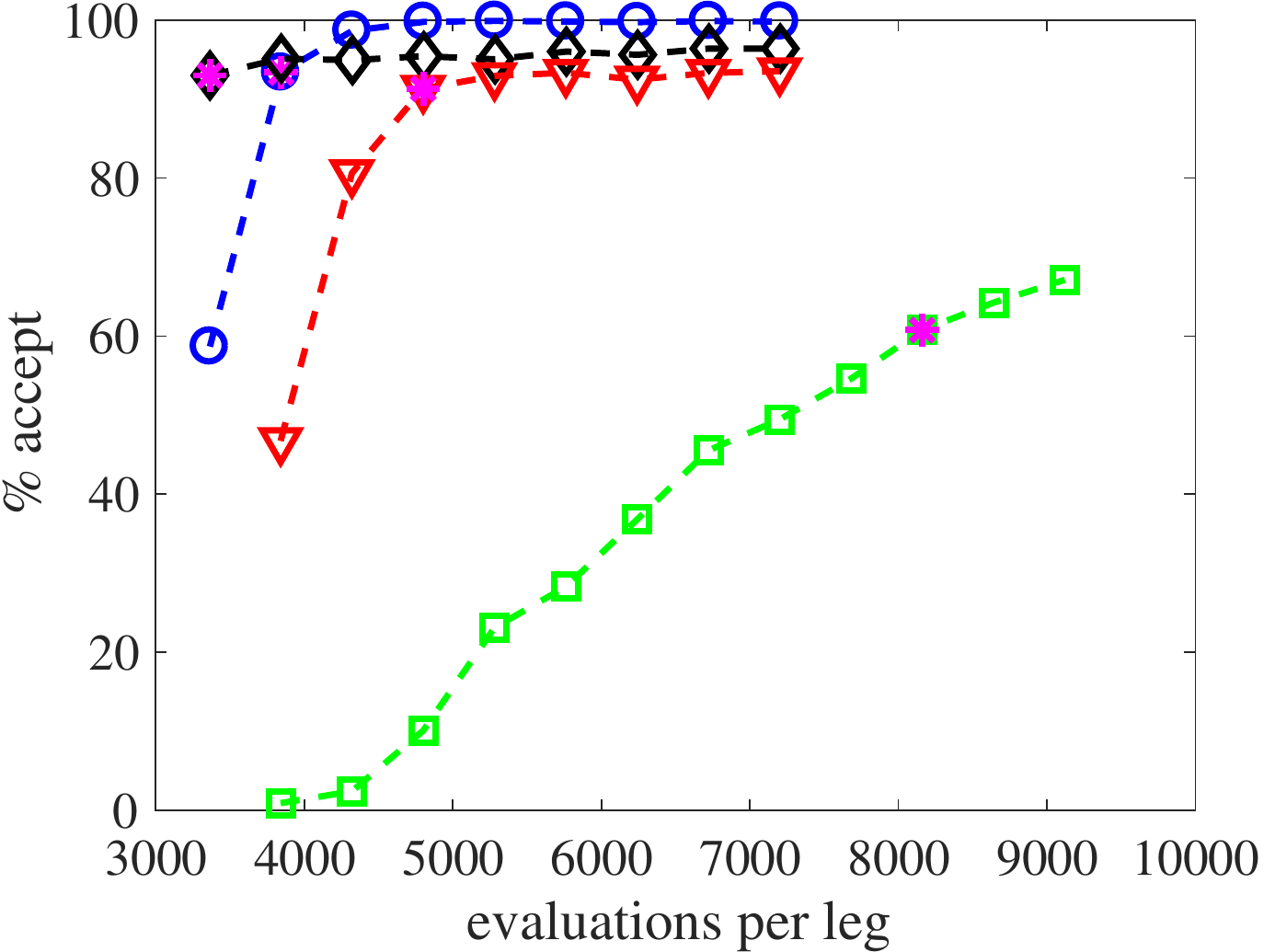}\quad
\includegraphics[width=0.48\hsize]{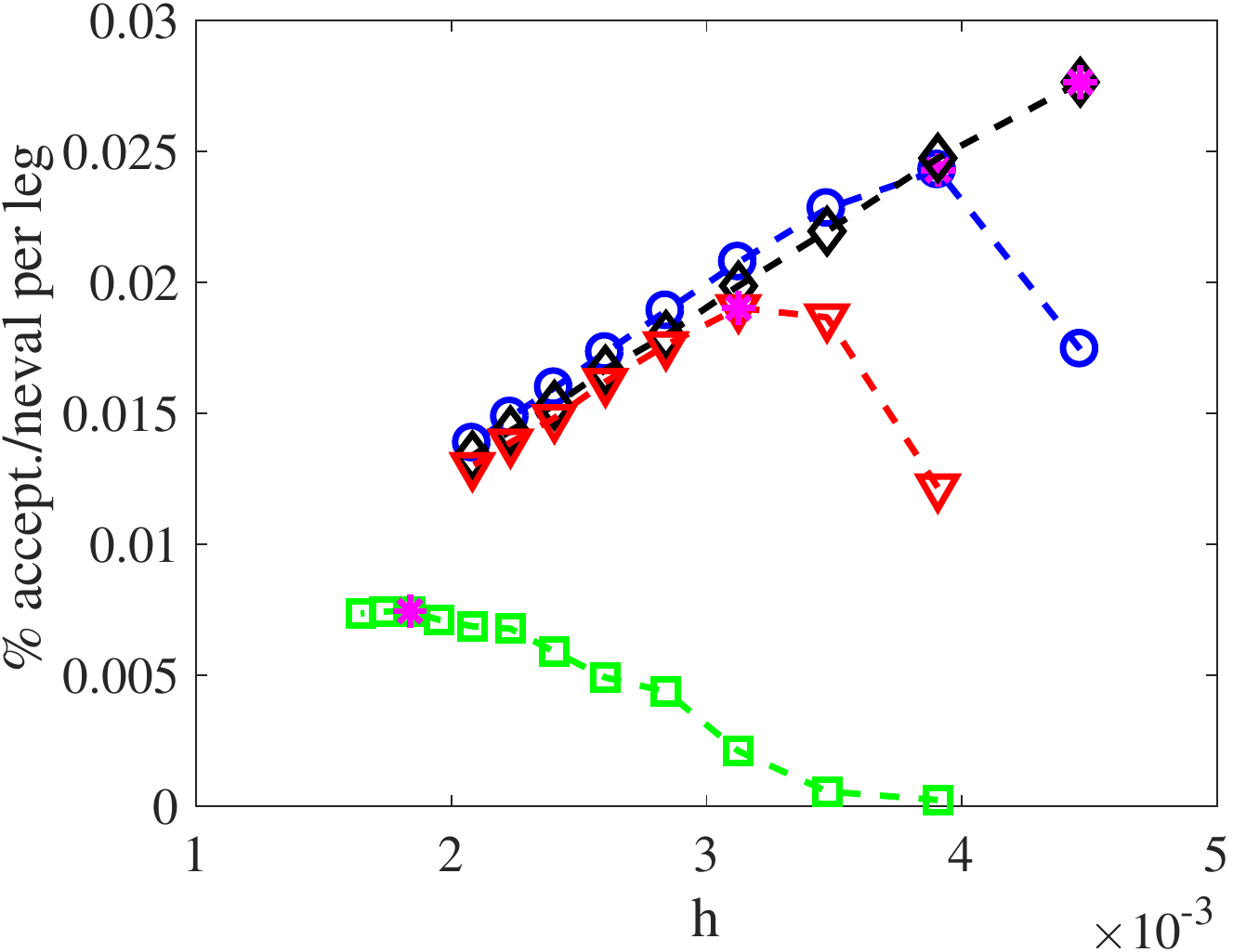}
\\
\includegraphics[width=0.48\hsize]{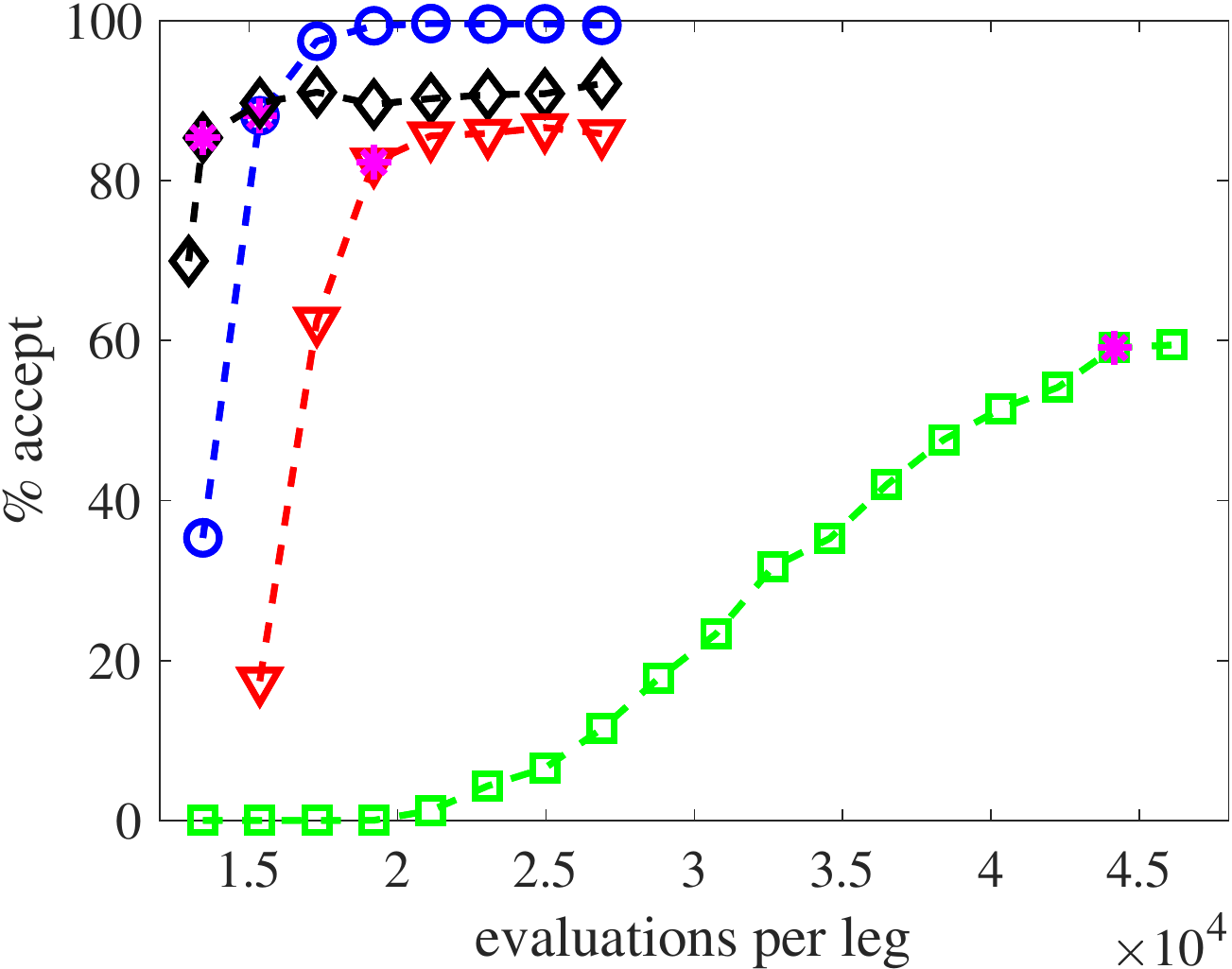}\quad
\includegraphics[width=0.48\hsize]{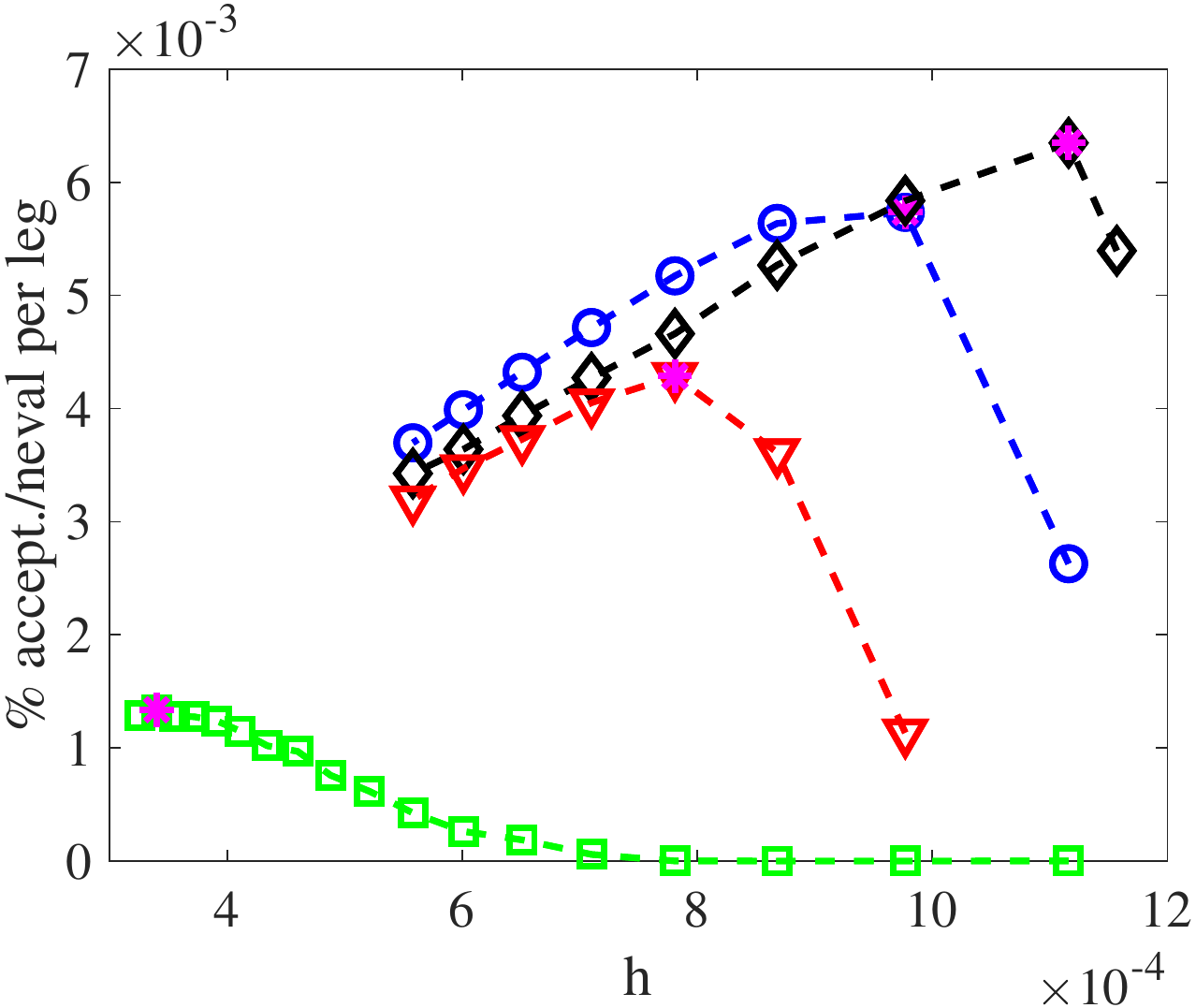}
\caption{\small Gaussian model. Comparison of the symmetrically processed methods with $b=0.348674$ (blue circles) and $b=0.340200$ (black diamonds) with the integrators BlCaSa (red triangles) and leapfrog (green squares). The top, middle and bottom panels have $d=256$,
$d=1024$ and $d=4096$ respectively. On the left,  acceptance percentage as a function of the number of evaluations per integration leg. On the right, acceptance percentage divided by number of gradient evaluations
for different values of the step size $h$. For each integrator and \(d\) a star symbol on the marker identifies the most efficient run. }
\label{Gaussian_efic}
\end{figure}

As in \cite{sergferyo} our first test problem is the model multivariate Gaussian target  with density (\(q_j\) is the \(j\)-component of \(q\))
\[
\propto\exp{\left ( -\frac{1}{2} \sum_{j=1}^d j^2 q^2_j \right )}.
%
\]
We have used the dimensions \(d=256\), \(d= 1024\)  and \(d=4096\): our interest is in problems of large dimensionality, those where efficiency is more important (for targets in very low dimension leapfrog/Verlet performs very satisfactorily as a consequence of its optimal linear stability properties \cite{acta}).
 As in \cite{daniel}, integrations were performed {{}in the interval \(0\leq t\leq 5\)}, for different stable choices of \(h\),  and we generated Markov chains of 5,000 elements initialized from the target distribution. {{} For reasons discussed in detail in \cite{daniel} and borne out by the extensive numerical experiments reported in that paper,}
  the efficiency of the algorithm and the \emph{quality of the samples} is entirely determined by the acceptance rate and therefore we will focus on this metric. The conclusions to be drawn as to the merit of the different integrators based on the behaviour of the acceptance rate are the \emph{same} that may be obtained by considering other metrics such as mean square displacement, effective sample sizes of the different components of \(q\), etc.

 Our results are summarized in Figure~\ref{Gaussian_efic}, where we compare the symmetrically processed integrators with $b=0.348674$  and $b=0.340200$  (see Table~\ref{table}) against the integrator BlCaSa and standard velocity leapfrog/Verlet. The intermediate values
$b=0.346660$  and $b=0.343684$ in the table were also run; the results  interpolate between those of $b=0.348674$  and those of $b=0.340200$ and are not reported so as to not blur the plots. For the symmetrically processed methods, the reported gradient evaluation count includes the evaluations required by the pre and postprocessing.

The \emph{left} subplots give, for the four integrators,  acceptance rate as a function of the number of gradient evaluations per integration leg. Of course, for each integrator, more gradient evaluations per leg (corresponding to smaller values of \(h\)) provide higher acceptance rates. The advantage of the three multistage integrators over Verlet is clearly borne out, and this advantage becomes more pronounced as the dimensionality increases (i.e.\ as  it becomes more important to have efficient algorithms). For the integrator with $b=0.348674$, the runs with more gradient evaluations deliver acceptance rates of virtually \(100\%\), which is in agreement with the low energy errors that correspond to the extremely low value of \(\|\rho\|_{\bar h}\) reported in the table. Also note that $b=0.340200$ operates very well for runs with fewer evaluations (larger \(h\)) and this matches the fact that this method was derived using a larger value of \(\bar h\).

Even though very low acceptance rates are unwelcome, very high acceptance rates are also undesirable. In fact, it is well understood \cite{optimal,acta,daniel} that, for a given integrator, a very high acceptance rate signals that the value of \(h\) being employed is too small: one would do better by using the available computational budget to obtain longer Markov chains by getting proposals at a lower computational cost per leg, even if that implies rejecting more proposals.
For this reason, it is not easy to assess the efficiency of the different integrators by examining the left subplots we have been discussing.
This efficiency  is best assessed from the \emph{right} subplots that give, for different values of \(h\), the acceptance rate \emph{per unit computational cost}, i.e. the result of dividing the empirical acceptance rate achieved in a simulation by the number of gradient evaluations required. Larger values of this metric correspond to more efficient sampling. The figure shows that according to our discussion, for a given integrator, the best efficiency (i.e. the highest marker) is not obtained when \(h\) is too small or too large.  In the figure the most efficient run for each of the four integrators has been indicated by a purple star on the corresponding marker. The right panels make it clear that BlCaSa is far more efficient than Verlet and the gap in efficiency increases with the dimensionality. For \(d=4096\) the optimal value of \(h\) for Verlet is \(\approx 2\times 10^{-4}\) and then the acceptance rate divided by the number of gradient evaluations is \(\approx 1\times 10^{-3}\); for BlCaSA the optimal value of \(h\) is substantially larger
\(\approx 8\times 10^{-4}\) and yields an acceptance rate per unit computational cost  \(\approx 4\times 10^{-3}\), approximately a fourfold improvement on Verlet. In turn the performance of the symmetrically processed integrators clearly improves on BlCaSa. For
\(d=4096\), the integrator with $b=0.340200$ is roughly five times more efficient than Verlet and approximately 50\% more efficient than BlCaSa.

\begin{figure}[t]
\includegraphics[width=0.48\hsize]{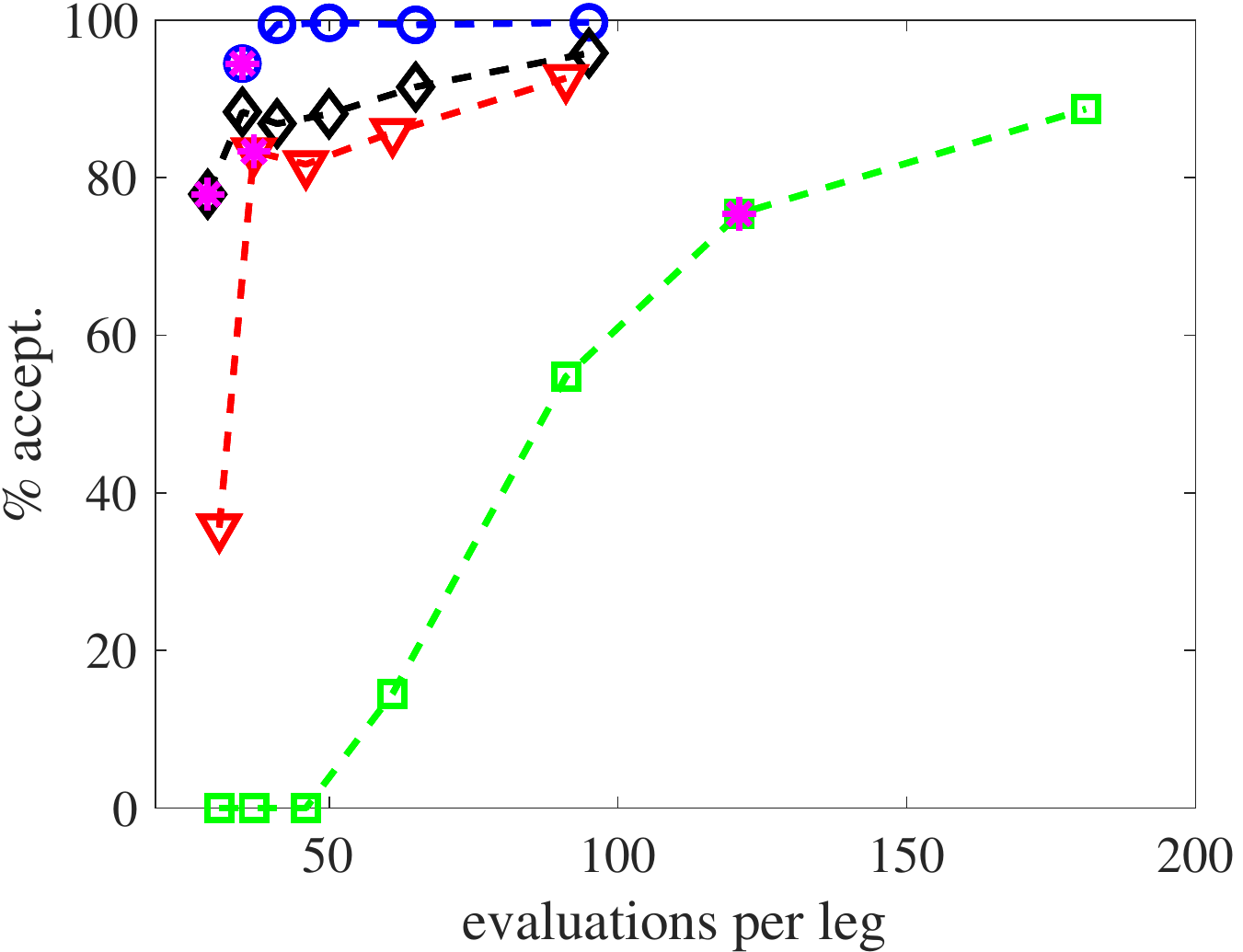}\quad
\includegraphics[width=0.48\hsize]{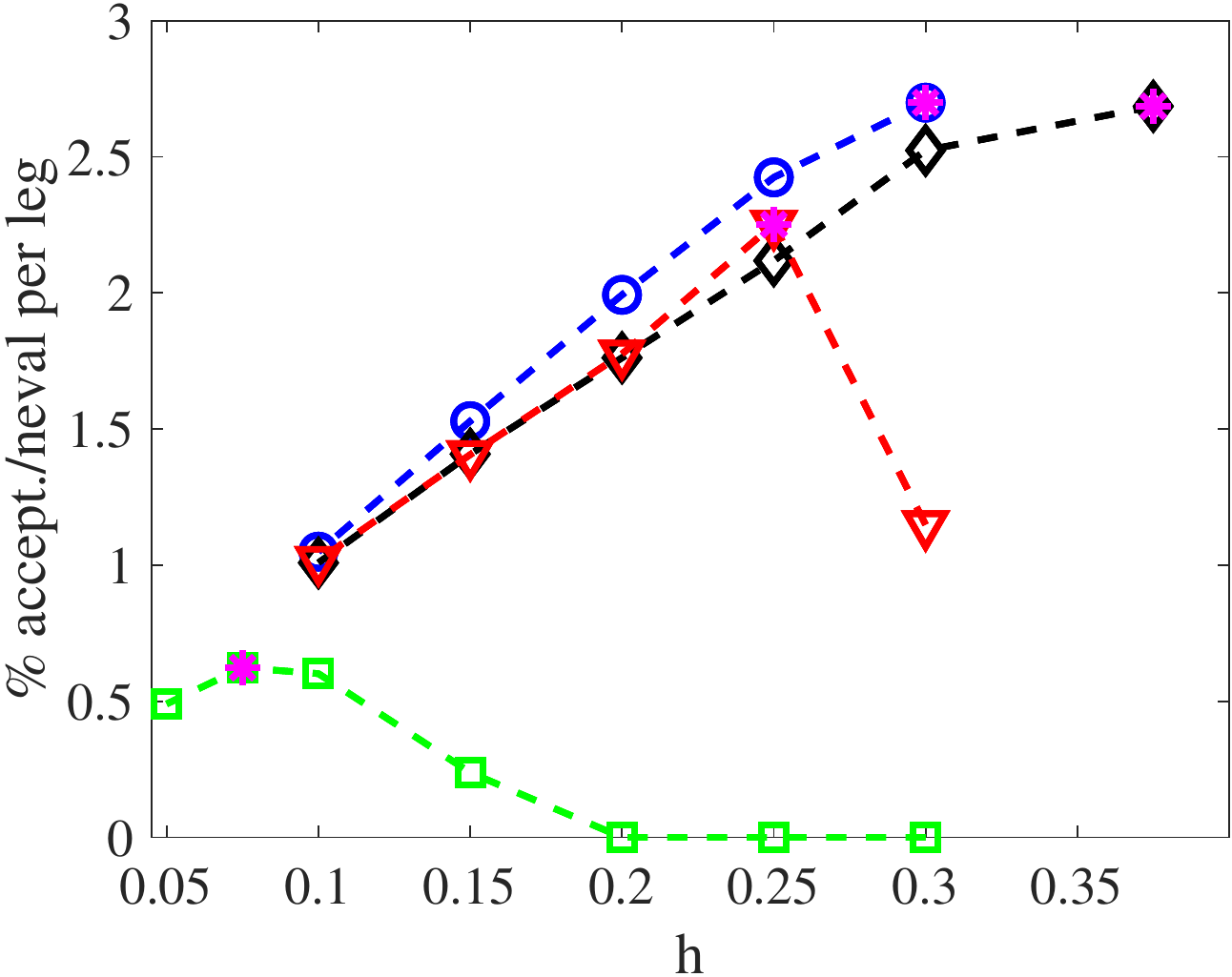}
\caption{\small Log-Gaussian Cox problem. Comparison of the symmetrically processed methods with $b=0.348674$ (blue circles) and $b=0.340200$ (black diamonds) with the integrators BlCaSa (red triangles) and leapfrog (green squares). On the left,  acceptance percentage as a function of the number of evaluations per integration leg. On the right, acceptance percentage divided by number of gradient evaluations
for different values of the step size $h$. For each integrator a star symbol on the marker identifies the most efficient run.}
\label{Cox_efic}
\end{figure}

We now present results (Figure~\ref{Cox_efic}) for a well-known Log-Gaussian Cox problem in Bayesian inference,  often used as a test \cite{christensen,girolami, daniel}. The dimension is \(d = 4096\). The  details of the sampling (length of the chain, burn-in, initialization, step sizes, time length of the integration legs, etc.) are exactly as in \cite{daniel} and will not be reproduced here. The general pattern of the results is very similar to that in Figure~\ref{Gaussian_efic}. Again BlCaSa is roughly four times more efficient than Verlet. Now $b=0.348674$ and $b=0.340200$ are equally efficient and their common efficiency is roughly \(25\%\) higher than that of BlCaSa.

{We also carried out experiments with the Boltzmann distribution of the Alkane molecule considered in \cite{cances} (see also \cite{cedric,daniel}). The results do not provide additional insights and will not be reported here. }

A final remark: since the number of gradient evaluations is  \(3N+1\) for BlCaSa and \(3N+5\) for the symmetrically processed integrators, the advantages of processing will decrease if the integration legs use very few time steps and become more marked when many time steps are taken.

\section{Other uses of symmetric modified processing}
\label{final}
As we will discuss now, HMC sampling is not the only application where symmetric modified processing may be useful.

In several problems, including the time-integration of parabolic partial differential equations or the Schr\"{o}dinger equation in imaginary time (as used in path integral computations),
 all the coefficients appearing in a splitting algorithm have to be positive (or complex with nonnegative real part). As is well known (see \cite{pos} and its references) this sets an upper bound of two to the order of accuracy that may be achieved by unprocessed splitting integrators. On the other hand, by using \emph{modified potentials}, it is possible \cite{chin,rowlands,taka} to construct kernels with effective order four and positive coefficients and this raises the question of how to construct suitable \emph{pre and post-processors with positive coefficients}.
Unfortunately,
for the standard processing format given by \eqref{eq:proc} and  \eqref{eq:pi},  the relations \eqref{eq:cons} imply that at least one \(c_i\) and one \(d_j\) have to be negative. This difficulty may be circumvented by symmetric modified processing as follows.

{{} Going back to the general setting of Section~\ref{quasi}, if \(N\geq 2\), the kernel \(\psi_h\) to be processed} is time reversible and we set \(\kappa_h = \psi_h\circ \pi_h\), then the definition in \eqref{eq:quasi} may obviously be rewritten as
\begin{equation} \label{eq:tpsi}
 \widetilde{\psi}_{N,h} = \kappa_h^\star \circ \psi_h^{N-2} \circ\kappa_h.
 \end{equation}
  {{}If now the map \(\kappa_h\) is sought in the form
 \[
 \kappa_h  = \phi_{\alpha_s h}^{A} \circ \phi_{\beta_s h}^{B}\circ \cdots \circ \phi_{\alpha_1 h}^{A} \circ \phi_{\beta_1 h}^{B},
 \]
 consistency demands
 \[\alpha_1 +\dots +\alpha_s  = 1,\qquad \beta_1 +\dots +\beta_s  = 1,
  \]
  and these relations may be satisfied with \(\alpha_i,\beta_i\geq 0\), \(i=1,\dots,s\).

  As an illustration of this idea, we have constructed a map \(\kappa_h\) to obtain via \eqref{eq:tpsi} a processed fourth-order integrator based on the well-known Rowlands method \cite{rowlands} for Hamiltonian systems of the form \eqref{eq:ham}. The Rowlands kernel uses modified potentials; it may be viewed as a modification of the standard velocity Verlet method where, in the kicks, the potential \(V(q)\) is replaced by a potential \(\widetilde V(q)\)
  obtained by adding to \(V(q)\) a multiple of
$\nabla V(q)^T M^{-1} \nabla V(q)$. More specifically, the Rowlands scheme reads
\begin{equation}  \label{rowlandsk}
  \psi_h = \widetilde{\phi}_{(\frac{1}{2},\frac{1}{48})h}^B \circ \phi_h^A \circ \widetilde{\phi}_{(\frac{1}{2},\frac{1}{48})h}^B,
\end{equation}
where $\widetilde{\phi}_{(b,c) h}^B$ denotes the $h$-flow of
\[
  \frac{d}{dt} q = 0,\qquad \frac{d}{dt} p = -\nabla \widetilde{V}_{(b,c)h}(q),
 \]
with
\[
   \widetilde{V}_{(b,c)h}(q) := b \, V(q) - h^2 \, c \, \nabla V(q)^T M^{-1} \nabla V(q).
\]
To achieve order four, a processor $\pi_h$ has to satisfy four order conditions (including the equation that ensures consistency). Those conditions are easily translated into as many conditions for the map $\kappa_h = \psi_h \circ \pi_h$. We looked for $\kappa_h$ of the form
\[
  \kappa_h = \phi_{\alpha_2 h}^{A} \circ \phi_{\beta_2 h}^{B}
   \circ \phi_{\alpha_1 h}^{A} \circ \widetilde{\phi}_{(\beta_1, \gamma_1) h}^{B},
\]
(note that one flow with modified potential is used), with five parameters at our disposal. The four polynomial  equations to be satisfied have the particular solution
\[
  \alpha_1 = \frac{6}{7}, \qquad \beta_1 =   \frac{23}{72}, \qquad \gamma_1 = \frac{55}{1728}, \qquad \alpha_2 =  \frac{1}{7}, \qquad
  \beta_2 =  \frac{49}{72},
\]
where all coefficients are positive.

Methods of this class, involving flows of modified potentials,
 can also be used for  Hamiltonian Monte Carlo
simulations in the context of lattice field theory \cite{kennedy09fgi,kennedy13shp}.}

 \bigskip

{\bf Acknowledgement.}  S.B., F.C. and J.M.S.-S. have been supported by project
  PID2019-104927GB-C21 (AEI/FEDER, UE). M.P.C. has been supported by projects PID2019-104927GB-C22 (GNI-QUAMC), (AEI/FEDER, UE)
 VA105\-G18 and VA169P20 (Junta de Castilla y Leon, ES) co-financed by
FEDER funds.

 \end{document}